\documentclass[12pt,cd]{amsart}%
\usepackage{amsfonts, amssymb,amsmath, amsthm, amsxtra, latexsym,  amscd}
\usepackage{times, float}
\usepackage{graphics}
\usepackage{hyperref}
\usepackage[latin1]{inputenc}
\usepackage{color}
\usepackage{amsmath}
\usepackage{amsfonts}
\usepackage{amssymb}
\usepackage{graphicx}%
\providecommand{\U}[1]{\protect\rule{.1in}{.1in}}
\theoremstyle{plain}
\newtheorem{theo+}{Theorem}[section]
\newtheorem{prop+}[theo+]{Proposition}
\newtheorem{coro+}[theo+]{Corollary}
\newtheorem{lemm+} [theo+]{Lemma}
\newtheorem{deep+}  [theo+]  {Deep Result}
\newtheorem{fact+}  [theo+]  {Fact}
\theoremstyle{definition}
\newtheorem{exam+}  [theo+]  {Example}
\newtheorem{rema+}  [theo+]  {Remark}
\newtheorem{defi+}  [theo+]  {Definition}
\newtheorem{xca+}[theo+]{Exercise}
\newenvironment{theorem}{\begin{theo+}}{\end{theo+}}
\newenvironment{proposition}{\begin{prop+}}{\end{prop+}}

\newenvironment{lemma}{\begin{lemm+}}{\end{lemm+}}

\newenvironment{definition}{\begin{defi+}}{\end{defi+}}

\numberwithin{equation}{section}

\def\Norm#1_#2{\Vert#1\Vert_{#2}}

\def\1{{\bf 1}}

\begin{document}
\title[Capelli-type operators]{The Capelli identity and Radon transform for Grassmannians }
\author{Siddhartha Sahi}
\address{Department of Mathematics, Rutgers University, Brunswick, NJ, USA }
\email{\textit{siddhartha.sahi@gmail.com }}
\author{Genkai Zhang}
\address{Mathematical Sciences, Chalmers University of Technology and Mathematical
Sciences, G\"oteborg University, SE-412 96 G\"oteborg, Sweden}
\email{\textit{genkai@chalmers.se}}
\thanks{Research by G. Zhang partially supported by the Swedish Science Council (VR)}

\begin{abstract}
We study a family $C_{s,l}$ of Capelli-type invariant differential operators
on the space of rectangular matrices over a real division algebra. The
$C_{s,l}$ descend to invariant differential operators on the corresponding
Grassmannian, which is a compact symmetric space, and we determine the image
of the $C_{s,l}$ under the Harish-Chandra homomorphism. We also obtain
analogous results for corresponding operators on the non-compact duals of the
Grassmannians, and for line bundles. As an application we obtain a Radon
inversion formula, which generalizes a recent result of B. Rubin for real Grassmannians.

\end{abstract}
\maketitle

\section{Introduction}

Let $\mathbb{F}=\mathbb{R},\mathbb{C},\mathbb{H}$ be a division algebra of
real dimension $d=1,2,4$. The first main result of this paper is a new
Capelli-type identity for Grassmannians over $\mathbb{F}$. This generalizes
earlier work of one of the authors \cite{sahi-2013-cap}, which solved a
problem posed by Howe-Lee \cite{Howe-Lee} for $\mathbb{F}=$ $\mathbb{R}$. Let
$r$ be an integer such that $1\leq r\leq n/2,$ and consider the three spaces
\begin{equation}
W=W_{n,r}=Mat_{n\times r}\left(  \mathbb{F}\right)  ,\;X=X_{n,r}=Mat_{n\times
r}^{\times}\left(  \mathbb{F}\right)  ,\;Y=Y_{n,r}=Gr_{n,r}\left(
\mathbb{F}\right)  \label{=WXY}%
\end{equation}
consisting, respectively, of all $\mathbb{F}$-matrices of shape $n\times r$,
the open subset of matrices of $\mathbb{F}$-rank $r$, and the Grassmannian of
$r$-dimensional $\mathbb{F}$-subspaces of $\mathbb{F}^{n}$. We write
$G_{n}=GL_{n}\left(  \mathbb{F}\right)  $ and $K_{n}=U_{n}\left(
\mathbb{F}\right)  $ for its maximal compact subgroup. Then $G_{n}$ and
$G_{r}$ act naturally on the left and right of $W$ and the actions preserve
$X$. Moreover we have natural isomorphims
\begin{equation}
Y\approx X/G_{r}\approx K_{n}/\left(  K_{r}\times K_{n-r}\right)  . \label{=Y}%
\end{equation}

We write $w^{\dag}=\overline{w}^{T}$ and consider the following
functions/differential operators on $W$%
\begin{equation}
\Psi\left(  w\right)  =\operatorname{det}\left(  w^{\dag}w\right)  ,\quad
L=\partial\left(  \Psi\right)  ,\quad C_{s,l}=\Psi^{s+l}L^{l}\Psi^{-s},
\label{=Csl}%
\end{equation}
where $\partial$ is the isomorphism between polynomials and differential
operators on $W$ induced by the pairing $\left\langle w_{1},w_{2}\right\rangle
=\mathrm{tr}\left(  w_{1}^{\dag}w_{2}\right)  $. Then $C_{s,l}$ is
$K_{n}\times G_{r}$ invariant and by (\ref{=Y}) it can be regarded as an
element of the algebra $\mathbf{D}_{K_{n}}\left(  Y\right)  $ of $K_{n}%
$-invariant operators on $Y=K_{n}/\left(  K_{r}\times K_{n-r}\right)  $, which
is a compact symmetric space of type $BC_{r}$.

By a result of Harish-Chandra \cite{He2} one has an algebra isomorphism
\[
\eta:\mathbf{D}_{K_{n}}\left(  Y\right)  \longrightarrow\mathcal{P}\left(
\mathfrak{a}^{\ast}\right)  ^{\mathsf{W}}%
\]
where $\mathfrak{a}$ is a Cartan subspace of $Y$, $\mathsf{W}$ is the
restricted Weyl group, and $\mathcal{P}\left(  \mathfrak{a}^{\ast}\right)
^{\mathsf{W}}$ is the algebra of $\mathsf{W}$-invariant polynomial functions
on $\mathfrak{a}^{\ast}$. We identify $\mathfrak{a}^{\ast}\approx
\mathbb{C}^{r}$ by choosing a basis $\{e_{i}\}$ of $\mathfrak{a}^{\ast}$ such
that the restricted roots are $\left\{  \pm e_{i},\pm2e_{i},\pm{e_{i}\pm
e_{j}}\right\}  $, and we set
\begin{equation}
\rho_{j}=d\left(  n/2-j+1\right)  -1,\quad\rho=\left(  \rho_{1},\ldots
,\rho_{r}\right)  . \label{=rho}%
\end{equation}
Our first result is a formula for the image of $C_{s,l}$ under the
Harish-Chandra isomorphism.

\begin{theorem}
\label{ThA} We have $\eta\left(  C_{s,l}\right)  =c_{s,l}\left(  z\right)  $
where
\begin{equation}
c_{s,l}\left(  z\right)  =\prod\nolimits_{i=0}^{l-1}\prod\nolimits_{j=1}%
^{r}\left[  \left(  2s-\rho_{1}+2i\right)  ^{2}-z_{j}^{2}\right]  .
\label{=csl}%
\end{equation}

\end{theorem}

This is proved as Theorem \ref{Th:main1} below, and in sections
\ref{section-noncompact}, \ref{section-non-spherical} we prove analogous
results for the non-compact duals of the Grassmannians, and for non-spherical
line bundles.

For $\mathbb{F}=\mathbb{R}$ the operator $L=\partial\left(  \Psi\right)  $ is
the \emph{Cayley-Laplace} operator; see \cite{rubin-riesz} and the references
therein. Using Theorem \ref{ThA} one can give a new proof of several key
results of \cite{rubin-riesz}, and extend these results to $\mathbb{C}$,
$\mathbb{H}$. We intend to discuss this in a subsequent paper.

Theorem \ref{ThA} turns out to have a beautiful application to Radon
inversion. For $r\leq r^{\prime}$ the Radon transforms $R:Y_{n,r}\rightarrow
Y_{n,r^{\prime}}$ and $R^{\prime}:Y_{n,r^{\prime}}\rightarrow Y_{n,r}$ are
defined as follows%
\begin{equation}
Rf\left(  y^{\prime}\right)  =\int_{y\subset y^{\prime}}f\left(  y\right)
dy\text{,\quad}R^{\prime}f\left(  x\right)  =\int_{y^{\prime}\supset
y}f\left(  y^{\prime}\right)  dy^{\prime}, \label{=Rad}%
\end{equation}
where $dy,$ $dy^{\prime}$ are certain natural invariant measures. We now
assume that
\begin{equation}
\text{(i) }r\leq r^{\prime}\leq n-r\text{,\quad(ii) }\frac{d}{2}(r^{\prime
}-r)\in\mathbb{Z}\text{.} \label{=i-ii}%
\end{equation}

\begin{theorem}
\label{ThB} Under the conditions (\ref{=i-ii}) we have $\left(  SR^{\prime
}\right)  R=I$ where
\[
S=\frac{1}{c_{\beta,l}\left(  \rho\right)  }C_{\beta,l},\quad\beta=\frac{d}%
{2}(n-r^{\prime}),\quad l=\frac{d}{2}\left(  r^{\prime}-r\right)  .
\]

\end{theorem}

This is proved as Theorem \ref{Th: main2} below.

We refer the reader to \cite{GGG, He1, He2} for background on the Radon
transform and its modern interpretation in the context of homogeneous spaces
and integral geometry. We also refer the reader to \cite{GGR}, \cite{GK},
\cite{Grinberg-jdg}, \cite{grinberg-rubin-ann}, \cite{Kakehi},
\cite{Oshima-capelli}, \cite{Oda-Oshima} for other examples of Radon inversion
formulas, and to \cite{AB}, \cite{AGS}, \cite{Olaf-Pasq-Rub},
\cite{gz-radon-imrn} for the closely related cosine transform. For
$\mathbb{F}=\mathbb{R}$, Rubin \cite[Th 8.2]{rubin-funk} has recently obtained
an inversion formula for the Funk transform on Stiefel manifolds, which is
essentially equivalent to Theorem \ref{ThB}, as explained in Remark
\ref{RubinFunk} below. Thus our result constitutes a generalization of Rubin's result.

We refer the reader to \cite{Howe-inv-thy, HU}, \cite{KostantSahiAdv,
Kostant-Sahi-inv} for background and perspective on the Capelli identity, to
\cite{Sahi-cap-unitary}, \cite{Sahi-unitary Shilov} for applications to
unitary representations, and to \cite{Knop-Sahi}, \cite{Sahi-interp} for the
connection with Jack polynomials and Macdonald polynomials. The papers
\cite{KostantSahiAdv, Kostant-Sahi-inv} emphasize the point of view that the
classical Capelli identity can be understood as the computation of the
eigenvalues of a certain invariant differential operator, or equivalently the
computation of its Harish-Chandra image. This is the perspective adopted in
\cite{sahi-spec, sahi-2013-cap} and also in the present paper.

The paper is organized as follows. In section \ref{section-Matrix space} we
introduce the necessary background, including a precise description of the
differential operators $C_{s,l}$ and the module structure of the space
$\mathcal{R}$ of $G_{r}$-invariant and $K$-finite functions on $X$. In section
\ref{section-Capelli identity} we prove our first main result, Theorem
\ref{Th:main1}. This result is extended in sections \ref{section-noncompact}
and \ref{section-non-spherical} to the non-compact and non-spherical settings,
respectively. Finally in section \ref{section-Radon} we recall some basic
facts about the Radon transform and prove our second main result, Theorem
\ref{Th: main2}.

We now discuss briefly the main ideas behind the proofs of our results.

To compute the Harish-Chandra image $c_{s,l}=\eta\left(  C_{s,l}\right)  $, we
follow the approach of Kostant-Sahi \cite{KostantSahiAdv},
\cite{Kostant-Sahi-inv}, \cite{sahi-spec}. Let $\mathcal{R}$ be the algebra of
$K_{n}$-finite functions on $Y=K_{n}/\left(  K_{r}\times K_{n-r}\right)  $.
The irreducible $K_{n}$-submodules ($K_{n}$-types) of $\mathcal{R}$ occur with
multiplicity $1$, and are uniquely determined by their highest weights $\mu
\in\mathfrak{a}^{\ast}\approx\mathbb{C}^{r}$. On each such $K_{n}$-type the
operator $C_{s,l}$ acts by the scalar $c_{s,l}\left(  \mu+\rho\right)  $,
where $\rho$ is the half sum of positive roots. The operator $C_{s}=C_{s,1}$
vanishes on certain $K_{n}$-types and we show that these vanishing conditions
suffice to \emph{characterize} the polynomial $c_{s}=c_{s,1}$ up to an overall
scalar, which is then determined by an auxiliary computation. This proves
Theorem \ref{Th:main1} for $C_{s,1}$ and the result for $C_{s,l}$ follows by a
factorization argument.

We consider also the non-compact symmetric space $D$ dual to the Grassmannian
$Y$ and study corresponding operators $\tilde{C}_{s,l}$. The space $D$ can be
realized as an open domain in $Y$. The eigenvalues of $\tilde{C}_{s,l}$ on the
Harish-Chandra spherical functions can be found by using a well-known
observation \cite{He2, He3} that the spherical polynomials on $Y$ restricted
to $D$ are the Harish-Chandra functions. When $Y$ and $D$ are Hermitian
symmetric compact or non-compact spaces, the operators $\tilde{C}_{s,l}$ also
act on sections of homogeneous line bundles, and we find the corresponding eigenvalues.

The main idea behind the proof of Theorem \ref{Th: main2} is straightforward
-- one simply compares the eigenvalues of $R^{\prime}R$ to those of $C_{s,l}$.
The eigenvalues of $R^{\prime}R$ are computed explicitly in
\cite{Grinberg-jdg} for $\mathbb{F}=\mathbb{C}$ and stated without proof for
$\mathbb{R},\mathbb{H}$. However the statement in \cite[Theorem 6.2]%
{Grinberg-jdg} for $\mathbb{H}$ is not correct, and it was not clear to us how
to extend the ideas of \cite{Grinberg-jdg} to obtain the right formula. Thus
in this paper we give a completely different argument using ideas from
\cite{gz-radon-imrn, gz-radon-bsd}, where the Radon transform is related to
the spherical transform of the sine functions on Grassmannians, and its
eigenvalues are related to special values of spherical polynomials. The
results of \cite{gz-radon-imrn, gz-radon-bsd} are proved under certain
restrictions on $r,n$ and involve certain undetermined constants. In this
paper we complete these arguments by relaxing the restrictions and explicitly
determining the relevant constants. This leads to the proof of Theorem
\ref{Th: main2}.

\textbf{Acknowledgement. }We thank Eitan Sayag for suggesting a connection
between the Capelli identity and the Radon transform, and for helpful
discussions at an early stage of this project. We also thank Semyon Alesker,
Sigurdur Helgason, Toshiyuki Kobayashi, Toshio Oshima, and Boris Rubin for
many helpful comments and suggestions.

\section{Matrix spaces and Grassmannians\label{section-Matrix space}}

\subsection{Matrix spaces and Jordan algebras}

For $\mathbb{F=R}$, $\mathbb{C}$, $\mathbb{H}$, let $A=A_{r}(\mathbb{F})$
denote the space $r\times r$ Hermitian $\mathbb{F}$-matrices
\[
A=\left\{  u\in Mat_{r\times r}\left(  \mathbb{F}\right)  \mid u^{\dag
}=u\right\}  \text{.}%
\]
Then $A$ is a Euclidean Jordan algebra, and we write $e$ for the identity
element and $\operatorname{tr}$ and $\operatorname{det}$ for the trace form
and the Jordan norm polynomial. We refer the reader to \cite{FK-book} for
basic results and terminology related to Jordan algebras.

\begin{lemma}
For $u\in A$ let $D_{u}$ denote the directional derivative, then we have
\begin{equation}
\left(  D_{u}\operatorname{det}\right)  (e)=\operatorname{tr}(u).
\label{eq:def-tr}%
\end{equation}

\end{lemma}

\begin{proof}
See \cite[Proposition III.4.2]{FK-book}.
\end{proof}

The structure group of $A$ is $G_{r}=GL_{r}\left(  \mathbb{F}\right)  $, which
acts on $A$ by $g:u\mapsto gug^{{\dagger}}$. The automorphism group is
$K_{r}=U\left(  r,\mathbb{F}\right)  $, which is the stabilizer of $e$. The
structure group preserves the determinant up to a scalar multiple; thus we
have%
\begin{equation}
\operatorname{det}(gug^{{\dagger}})=\nu(g)\operatorname{det}(u), \label{eq:nu}%
\end{equation}
where $\nu$ is a certain character of $G_{r}$.

Let $W=Mat_{n\times r}(\mathbb{F})$ be the space of $n\times r$ matrices,
where we assume as before that%
\[
r\leq n-r.
\]
We have a natural map $Q:W\rightarrow A$,%
\[
Q\left(  w\right)  =w^{{\dagger}}w.
\]
By polarization we get a positive definite inner product $\left\langle
\cdot,\cdot\right\rangle $ on $W$ satisfying%
\[
\left\langle w,w\right\rangle =\operatorname{tr}(w^{{\dagger}}w),
\]
the trace being computed for real linear transformations. This gives us an
isomorphism between polynomials and constant coefficient differential
operators
\begin{equation}
\partial:\mathcal{P}\left(  W\right)  \approx\mathcal{D}\left(  W\right)
\label{=ptl}%
\end{equation}

Let $X\subset W$ be the open subset of matrices of rank $r$, then $Q\left(
X\right)  $ is the positive cone of $A$. Moreover if%
\begin{equation}
x_{0}=%
\genfrac{[}{]}{0pt}{}{I_{r}}{0}
\label{=x0}%
\end{equation}
then $x_{0}\in X$ and $Q\left(  x_{0}\right)  =e.$ The group $G_{n}\times
G_{r}$ acts naturally on $W$, and the polynomial
\begin{equation}
\Psi(w)=\operatorname{det}Q\left(  w\right)  =\operatorname{det}\left(
w^{{\dagger}}w\right)  \label{=Psi}%
\end{equation}
transforms under $K_{n}\times G_{r}$ as follows:
\begin{equation}
\Psi(kwg)=\operatorname{det}(g^{\dag}w^{{\dagger}}k^{{\dagger}}%
kwg)=\operatorname{det}(g^{\dag}w^{{\dagger}}wg)=\nu(g^{{\dagger}%
})\operatorname{det}(w^{{\dagger}}w)=\nu(g)\Psi(w). \label{=Psikg}%
\end{equation}

\subsection{Grassmannians and invariant differential operators}

The Grassmannian $Y=Gr_{n,r}(\mathbb{F})$, consisting of $r$ dimensional
subspaces of $\mathbb{F}^{n}$, has several different realizations that will
play a role below. First, we have a $G=G_{n}$ equivariant map%
\begin{equation}
\operatorname{col}:X\rightarrow Y, \label{=col}%
\end{equation}
where $\operatorname{col}\left(  x\right)  $ is the column space of $x$. This
descends to a homeomorphism
\begin{equation}
X/G_{r}\approx Y, \label{=XGr}%
\end{equation}
which realizes $X$ as the principal $G_{r}$-bundle associated to the
tautological bundle on $Y$. Also, the group $K=K_{n}$ acts transitively on $Y$
and the stabilizer in $K$ of
\[
y_{0}=\operatorname{col}\left(  x_{0}\right)
\]
is the symmetric subgroup $M=K_{r}\times K_{n-r}$. This realizes $Y$ as
compact symmetric space%
\[
Y=K/M.
\]
Finally, the stabilizer of $y_{0}$ in $G$ is a maximal parabolic subgroup $P$,
thus we get%
\begin{equation}
Y=G/P. \label{=GP}%
\end{equation}
This realization is of importance when studying principal series
representations of $G$, \cite{Sahi-crelle, gkz-ma}.

Let $C^{\infty}(X)^{G_{r}}$ be the space of $G_{r}$-invariant smooth functions
on $X$; by (\ref{=XGr}) we get a $G$-equivariant isomorphism
\[
C^{\infty}(X)^{G_{r}}\approx C^{\infty}(Y).
\]
More generally if $\nu$ is as in (\ref{eq:nu}) and $s\in\mathbb{C}$ we can
consider the space of $G_{r}$-equivariant functions%
\[
C^{\infty}(X)^{G_{r},s}=\{f\in C^{\infty}(X):f(xg)=\nu(g)^{s}f(x)\}.
\]
Note that $\Psi(x)=\operatorname{det}\left(  x^{{\dagger}}x\right)  $ is
positive on $X$ and hence $\Psi^{s}$ is a well-defined function on $X$. Now
(\ref{=Psikg}) implies that multiplication by $\Psi^{s}$ gives a $K_{n}%
$-equivariant isomorphism%
\[
f\mapsto\Psi^{s}f:C^{\infty}(X)^{G_{r},s}\approx C^{\infty}(X)^{G_{r}}\approx
C^{\infty}(Y)\text{.}%
\]

We now introduce the differential operators which play a key role in this paper.

\begin{definition}
With $\partial$ as in (\ref{=ptl}), $s\in\mathbb{C}$, and $l\in\mathbb{N}$, we
define
\begin{equation}
L=\partial(\Psi),\quad C_{s,l}=\Psi^{s+l}L^{l}\Psi^{-s},\quad C_{s}=C_{s,1}
\label{eq:C}%
\end{equation}

\end{definition}

Clearly $C_{s,l}$ is a $K\times G_{r}$-invariant differential operator on
$C^{\infty}(X)$. Now by a standard argument \cite[Chapt. II, Sect. 3]{He2}
applied to (\ref{=XGr}), we conclude that $C_{s,l}$ defines a $K$-invariant
differential operator on $Y$%
\[
C_{s,l}:C^{\infty}(Y)\rightarrow C^{\infty}(Y).
\]
More generally if $t$ is another complex number, then
\[
C_{s,l}:C^{\infty}(X)^{G_{r},t}\rightarrow C^{\infty}(X)^{G_{r},t}.
\]

\subsection{$K$-types on the Grassmannian}

\label{sec-K}

To study the action of the operators $C_{s,l}$ it is convenient to pass to an
algebraic setting. We write $\mathcal{P}$ for the algebra of polynomial
functions on $W$, and for each positive integer $m$ we define%
\[
\mathcal{P}^{m}=C^{\infty}(X)^{G_{r},m}\cap\mathcal{P}.
\]
Also let $\mathcal{R}$ be the subspace of $C^{\infty}(X)^{G_{r}}$ consisting
of $K$-finite functions. The operators $C_{s,l}$ preserve the spaces
$\mathcal{R}$ and $\mathcal{P}^{m}$, and we now describe their $K$-module structures.

As noted above, the Grassmannian $Y=K/M$ is a compact symmetric space, and we
fix some notation relevant to this structure. Let $\mathfrak{k}$ and
$\mathfrak{m}$ denote the complexified Lie algebras of $K=K_{n}$ and
$M=K_{r}\times K_{n-r}$ , and fix a Cartan decomposition and Cartan
subalgebra
\begin{equation}
\mathfrak{k}=\mathfrak{m}+\mathfrak{p,\quad h=t}+\mathfrak{a} \label{=Cartan}%
\end{equation}
where $\mathfrak{a\subset p}$ is a Cartan subspace and $\mathfrak{t}$ is the
centralizer of $\mathfrak{a}$ in $\mathfrak{m}$. Since $r\leq n-r$ by
assumption, the restricted root system $\Sigma\left(  \mathfrak{a}%
,\mathfrak{k}\right)  $ of type $BC_{r}$. We fix a basis $\{e_{i}\}$ of
$\mathfrak{a}^{\ast}$ such that the positive restricted roots $\alpha$ and
their multiplicities $m_{\alpha}$ are given in terms of $d=\dim\mathbb{F}$ as
follows
\[%
\begin{tabular}
[c]{|c|c|c|c|}\hline
$\alpha$ & $e_{i}$ & $2e_{i}$ & ${e_{i}\pm e_{j}}$\\\hline
$m_{\alpha}$ & $d\left(  n-2r\right)  $ & $d-1$ & $d$\\\hline
\end{tabular}
\ \ \ \text{.}%
\]
The half sum of the positive roots is $\rho=\sum_{j=1}^{r}\rho_{j}e_{j}$
where
\begin{equation}
\rho_{j}=\frac{1}{2}\left[  2d(r-j)+2(d-1)+d(n-2r)\right]  =d\left(
n/2-j+1\right)  -1. \label{rho}%
\end{equation}

Let $\Lambda$ be the set of even partitions of length $\leq r$:
\begin{equation}
\Lambda=\left\{  \left(  \mu_{1},\ldots,\mu_{r}\right)  \in\left(
2\mathbb{Z}\right)  ^{r}:\mu_{1}\geq\cdots\geq\mu_{r}\geq0\right\}  .
\label{=Lam}%
\end{equation}
Then $\Lambda$ parametrizes the set $\left(  K/M\right)  ^{\wedge}$ of
equivalence classes of irreducible $M$-spherical representations of $K$ as
follows. We identify $\mu\in$ $\Lambda$ with the $\mathfrak{h}$-weight that
vanishes on $\mathfrak{t}$ and restricts to $\sum\mu_{j}e_{j}$ on
$\mathfrak{a}$. For $\mathbb{F}=\mathbb{C}$ and $\mathbb{H}$ the group
$K=U(n,\mathbb{F)}$ is connected, and we write $V_{\mu}$ for the irreducible
$K$-module with highest weight $\mu\in\Lambda$. If $\mathbb{F}=\mathbb{R}$
then $K$ is the disconnected group $O\left(  n,\mathbb{R}\right)  $ and we let
$\tilde{V}_{\mu}$ denote the irreducible representation of $K_{o}=SO\left(
n,\mathbb{R}\right)  $ with highest weight $\mu$. Then $\tilde{V}_{\mu}$
extends uniquely to an $M$-spherical representation $V_{\mu}$ of $K$; however
if $r=n-r$ then we define%
\[
V_{\mu}=ind_{K_{o}}^{K}\left(  \tilde{V}_{\mu}\right)  .
\]
By the Cartan-Helgason theorem, the map $\mu\mapsto V_{\mu}$ gives a bijection
between $\Lambda$ and $\left(  K/M\right)  ^{\wedge}$. The following result is standard.

\begin{lemma}
\label{Lem:Rm}As a $K$-module, the algebra $\mathcal{R}$ admits a
multiplicity-free direct sum decomposition $\mathcal{R}=\oplus_{\mu\in\Lambda
}\mathcal{R}_{\mu}$ where $\mathcal{R}_{\mu}\approx V_{\mu}$.
\end{lemma}

We need a similar description of $\mathcal{P}^{m}$. For this we define%
\begin{equation}
\Lambda^{m}=\left\{  \mu\in\Lambda\mid\mu_{1}\leq2m\right\}  . \label{=Lam-m}%
\end{equation}

\begin{lemma}
\label{Lem: Rmm} We have a decomposition $\mathcal{P}^{m}=\oplus_{\mu
\in\Lambda^{m}}\mathcal{P}_{\mu}^{m}$, where $\mathcal{P}_{\mu}^{m}\approx
V_{\mu}.$
\end{lemma}

\begin{proof}
This is proved in \cite[Proposition 6.4]{gz-radon-imrn} for $\mathbb{F}%
=\mathbb{R}$ and $\mathbb{H}$ and in \cite[Theorem 5.3]{gkz-JLT} for
$\mathbb{F}=\mathbb{C}$. It also follows from \cite{Kostant-Sahi-inv}.
\end{proof}

\section{The Capelli identity for
Grassmannians\label{section-Capelli identity}}

In this section we find the spectrum of the operator $C_{s,l}$ on
$\mathcal{R}$, the main result being Theorem \ref{Th:main1} below. The key
computation involves the special case $l=1$. Thus we consider the operator%
\[
C_{s}:=C_{s,1}=\Psi^{s+1}L\Psi^{-s}.
\]
and we write $c_{s,\mu}$ for its eigenvalue on $\mathcal{R}_{\mu}$. We first
prove that $c_{s,\mu}$ is a polynomial in $s$ and $\mu$, and we determine its
leading term in $s$. For any $\mu$ we denote its $\rho$-shift by $\bar{\mu
}=\mu+\rho$. Also let $\mathcal{W}_{r}$ denote the Weyl group of type $BC_{r}$
acting on polynomials in $r$ variables by sign changes and permutations.

\begin{lemma}
\label{Lem: p}There exists a polynomial $p(t,z)=p\left(  t,z_{1},\ldots
,z_{r}\right)  $ such that $c_{s,\mu}=p(s,\bar{\mu})$. Moreover $p$ has
$t$-degree $\leq2r$, total $z$-degree $\leq2r$, and is $\mathcal{W}_{r}%
$-invariant in $z$.
\end{lemma}

\begin{proof}
Fix $\phi$ in $\mathcal{H}(\mu)$ such that $\phi(y_{0})=1$. Viewed as a smooth
function on $X$, we have $\phi(x_{0})=1$. Evaluating the eigenvalue equation
$C_{s}\phi=c_{s,\mu}\phi$ at $x_{0}$ leads to
\[
c_{s,\mu}=C_{s}\phi(x_{0})=\Psi^{1-s}L(\Psi^{s}\phi)(x_{0})=L(\Psi^{s}%
\phi)(x_{0}).
\]
Since $L$ is a polynomial differential operator of degree $2r$ and $\Psi^{s}$
is a power of the polynomial function $\Psi$, we see by Leibniz rule that
$c_{s,\mu}$ is a polynomial of degree $\leq2r$ in $s$. The coefficient
$\kappa_{j}$ of $s^{j}$ depends only on $\mu$, or equivalently on $\bar{\mu}$,
and so we can write%
\begin{equation}
c_{s,\mu}=\sum\nolimits_{i\leq2r}\kappa_{i}\left(  \bar{\mu}\right)  s^{i}
\label{=csmu}%
\end{equation}
Consider the equations (\ref{=csmu}) for $2r$ distinct values of $s,$ say
$s=1,\ldots,2r$. By non-vanishing of the Vandermonde determinant, we can
invert this system of equations to obtain $\gamma_{ij}\in\mathbb{Q}$ such
that
\[
\kappa_{i}\left(  \bar{\mu}\right)  =\sum\nolimits_{j\leq2r}\gamma
_{ij}c_{j,\mu}\text{ }%
\]
for all $i$ and all $\mu$. By the Harish-Chandra homomorphism we can write
$c_{j,\mu}=p_{j}(\bar{\mu})$, where $p_{j}\left(  z\right)  $ is a
$\mathcal{W}_{r}$-invariant polynomial of total degree $\leq2r$. Then the
polynomials
\[
k_{i}\left(  z\right)  =\sum\nolimits_{j\leq2r}\gamma_{ij}p_{j}\left(
z\right)
\]
satisfy the same properties, and $p\left(  t,z\right)  =\sum\nolimits_{i\leq
2r}k_{i}\left(  z\right)  t^{i}$ satisfies the requirements of the lemma.
\end{proof}

\begin{lemma}
\label{Lem: lead} The leading term of $p\left(  t,z\right)  $ as a polynomial
in $t$ is $2^{2r}t^{2r}$.
\end{lemma}

\begin{proof}
First, it follows by chain rule and (\ref{eq:def-tr}) that
\[
(D_{w}\operatorname{det}(x^{{\dagger}}x))(x_{0})=\operatorname{tr}%
(x_{0}^{{\dagger}}w+w^{{\dagger}}x_{0})=2(x_{0},w).
\]
The result now follows from \cite[Lemma 6.4]{sahi-2013-cap}.
\end{proof}

We now prove a simple vanishing condition for the eigenvalues.

\begin{lemma}
\label{vanishing} If $\mu\in\Lambda$ and $m=\mu_{1}/2$ then $c_{-m,\mu}=0$.
\end{lemma}

\begin{proof}
The operator $J_{m}:f\mapsto\Psi^{-m}f$ maps $\mathcal{P}_{\mu}^{m}$
isomorphically onto $\mathcal{R}_{\mu}$. Thus it suffices to prove that
$C_{-m}J_{m}$ acts by $0$ on $\mathcal{P}_{\mu}^{m}$. To see this we note that%
\[
C_{-m}J_{m}=\left(  \Psi^{-m+1}L\Psi^{m}\right)  \Psi^{-m}=\Psi^{-m+1}L.
\]
Now $L$ maps $\mathcal{P}_{\mu}^{m}$ to $\mathcal{P}_{\mu}^{m-1}$, which is
$0$ by Lemma \ref{Lem: Rmm}.
\end{proof}

We deduce a divisibility property of $p\left(  t,z\right)  $.

\begin{lemma}
The polynomial $p\left(  t,z\right)  $ is divisible by $2t-\rho_{1}+z_{1}.$
\end{lemma}

\begin{proof}
It is enough to show that $p\left(  t,z\right)  $ vanishes on the hyperplane%
\[
2t-\rho_{1}+z_{1}=0,
\]
i.e. that $g\left(  z\right)  =p\left(  \frac{\rho_{1}-z_{1}}{2},z\right)  $
is identically $0$. For $z=\bar{\mu}$ by the previous lemma we get
\[
g\left(  \bar{\mu}\right)  =p\left(  -\mu_{1}/2,\bar{\mu}\right)  =c_{-\mu
_{1}/2,\mu}=0
\]
Since the set $\Lambda+\rho$ is Zariski dense, the polynomial $g$ must be
identically $0.$
\end{proof}

We now prove an explicit formula for $p\left(  t,z\right)  $. Define
\begin{equation}
q(t,z)=\prod\nolimits_{j=1}^{r}\left[  t^{2}-z_{j}^{2}\right]  . \label{def-q}%
\end{equation}

\begin{theo+}
\label{Th:Cs} We have $p\left(  t,z\right)  =q\left(  2t-\rho_{1},z\right)  $.
\end{theo+}

\begin{proof}
By the previous lemma $p\left(  t,z\right)  $ is divisible by $\left(
2t-\rho_{1}\right)  +z_{1}$, and hence by $\mathcal{W}_{r}$-invariance it is
divisible by all factors of the form $\left(  2t-\rho_{1}\right)  \pm z_{j}$.
In other words $p\left(  t,z\right)  $ is divisible by $q\left(  2t-\rho
_{1},z\right)  $. However by (\ref{def-q}) and Lemma \ref{Lem: lead} both
polynomials have the same leading $t$-coefficient $2^{2r}t^{2r}$. Therefore
they must be equal.
\end{proof}

For each positive integer $l$ let $q_{l}\left(  t,z\right)  $ denote the
polynomial
\[
q_{l}\left(  t,z\right)  =\prod\nolimits_{i=0}^{l-1}q\left(  t+2i,z\right)
=\prod\nolimits_{i=0}^{l-1}\prod\nolimits_{j=1}^{r}\left[  \left(
t+2i\right)  ^{2}-z_{j}^{2}\right]  .
\]
The main result of this section is the following.

\begin{theorem}
\label{Th:main1} The eigenvalue of $C_{s,l}$ on $\mathcal{R}_{\mu}$ is
$q_{l}\left(  2s-\rho_{1},\bar{\mu}\right)  .$
\end{theorem}

\begin{proof}
We can factorize $C_{s,l}=\Psi^{s+l}L^{l}\Psi^{-s}$ iteratively as follows:%
\[
C_{s,l}=\left(  \Psi^{s+l}L\Psi^{-s-l+1}\right)  \left(  \Psi^{s+l-1}%
L^{l-1}\Psi^{-s}\right)  =C_{s+l-1}C_{s,l-1}=\prod\nolimits_{i=0}^{l-1}%
C_{s+i}\text{.}%
\]
Therefore by Theorem \ref{Th:Cs} the eigenvalue of $C_{s,l}$ on $\mathcal{R}%
_{\mu}$ is%
\[
\prod\nolimits_{i=0}^{l-1}p\left(  s+i,\bar{\mu}\right)  =\prod\nolimits_{i=0}%
^{l-1}q\left(  2s+2i-\rho_{1},\bar{\mu}\right)  =q_{l}\left(  2s-\rho_{1}%
,\bar{\mu}\right)  \text{.}%
\]

\end{proof}

\section{Non-compact Grassmannians\label{section-noncompact}}

In this section we explain how to extend the previous results to the
non-compact duals of the Grassmannians. Let $J$ be the Hermitian form of
signature $\left(  r,n-r\right)  $ on $\mathbb{F}^{n}$ given by
\[
(v,w)_{J}=v^{\dag}Jw{,\quad}J=\operatorname{diag}(I_{n-r},-I_{r})\text{,}%
\]
and let $\widetilde{K}$ be its isometry group. Thus for $\mathbb{F}%
=\mathbb{R},\mathbb{C},\mathbb{H}$ we have, respectively,
\[
\widetilde{K}=O\left(  r,n-r\right)  ,\,U\left(  r,n-r\right)  ,\,Sp\left(
r,n-r\right)  \text{.}%
\]
Let $M=K_{r}\times K_{n-r}$ be as before; then $M$ is a symmetric subgroup of
$\widetilde{K}$ and the symmetric space $\widetilde{Y}=\widetilde{K}/M$ is the
non-compact dual of the Grassmannian $Y=K/M$.

Let $X$ be the space of $n\times r$ matrices of rank $r$ as before, and let
$\widetilde{X}$ be the open subset
\begin{equation}
\widetilde{X}=\left\{  z\in X:(\cdot,\cdot)_{J}\text{ is positive definite
on}\operatorname{col}\left(  z\right)  \right\}  . \label{z-def}%
\end{equation}
Assume $n-r\geq r$ as before, and consider the following differential operator
on $\widetilde{X}$,
\[
\tilde{C}_{s,l}=\widetilde{\Psi}^{l+s}\widetilde{L}^{l}\Psi^{-s}%
,\quad\widetilde{\Psi}=\operatorname{det}(z^{{\dagger}}Jz),\widetilde{L}%
=\partial(\widetilde{\Psi}).
\]
Now $\widetilde{Y}$ can be realized as the quotient $\widetilde{X}/G_{r}$,
analogous to the realization $X/G_{r}$ of $Y$ (\ref{=col})-(\ref{=XGr}). Thus
we can identify $C^{\infty}(\widetilde{Y})$ with the space $C^{\infty
}(\widetilde{X})^{G_{r}}$ of right $G_{r}$-invariant functions on
$\widetilde{X}$ and $\tilde{C}_{s,l}$ then descends to a $\widetilde{K}%
$-invariant differential operator on $\widetilde{Y}$. We explain how to
compute its spectrum.

The complexifications of $\widetilde{K}$ and $K$ are conjugate inside
$G_{\mathbb{C}}=GL_{n}\left(  \mathbb{F}\right)  _{\mathbb{C}}$; we have
\[
ad\left(  \sigma\right)  :K_{{\mathbb{C}}}\approx\widetilde{K}_{{\mathbb{C}}%
},\quad\sigma=diag(I_{r},\sqrt{-1}I_{n-r}),
\]
which follows by noting that $\sigma=\sigma^{\dag}$ and $\sigma^{\dagger
}J\sigma=I$. Moreover $ad\left(  \sigma\right)  $ fixes $M$, thus allowing us
to transfer the structure theoretic results (\ref{=Cartan}), (\ref{rho})
\textit{etc.} from $\mathfrak{k}$ to $\widetilde{\mathfrak{k}}$. In particular
$\widetilde{\mathfrak{a}}=ad\left(  \sigma\right)  \mathfrak{a}$ is a Cartan
subspace for $\widetilde{K}/M$, and $\widetilde{\mathfrak{a}}^{\ast}$ is the
space of Satake parameters for $M$-spherical functions on $\widetilde{K}$. For
$\lambda\in\mathfrak{a}^{\ast}$ we define $\tilde{\lambda}\in
\widetilde{\mathfrak{a}}^{\ast}$ by
\begin{equation}
\tilde{\lambda}=\left(  \lambda+\rho\right)  \circ ad\left(  \sigma
^{-1}\right)  \label{=lam-til}%
\end{equation}
and let $\phi_{\tilde{\lambda}}\in C^{\infty}(\widetilde{Y})$ be the
corresponding spherical function. Also as before put%
\[
q_{l}\left(  t,z\right)  =\prod\nolimits_{i=0}^{l-1}\prod\nolimits_{j=1}%
^{r}\left[  \left(  t+2i\right)  ^{2}-z_{j}^{2}\right]
\]

\begin{theorem}
\label{thm-ncpt} For all $\lambda\in\mathfrak{a}^{\ast}$ we have $\tilde
{C}_{s,l}\left(  \phi_{\tilde{\lambda}}\right)  =q_{l}(2s-\rho_{1}%
,\lambda+\rho)\phi_{\tilde{\lambda}}$.
\end{theorem}

\begin{proof}
By the Harish-Chandra homomorphism there exists a $\mathcal{W}_{r}$-invariant
polynomial $p$ on $\mathfrak{a}^{\ast}$ such that we have $\tilde{C}%
_{s,l}\left(  \phi_{\tilde{\lambda}}\right)  =p\left(  \lambda+\rho\right)
\phi_{\tilde{\lambda}}$ for all $\lambda\in\mathfrak{a}^{\ast}$. By Zariski
density $p$ is uniquely determined by its values on the set $\Lambda+\rho$,
and thus it suffices to prove the theorem for $\mu\in\Lambda$.

Let $\mathcal{R}$ be the algebra of $K$-finite functions in $C^{\infty}\left(
X\right)  ^{G_{r}}$ and let $\widetilde{\mathcal{R}}$ be the algebra of
$\widetilde{K}$-finite functions in $C^{\infty}\left(  \widetilde{X}\right)
^{G_{r}}$. Also let $\mathcal{Q}$, $\widetilde{\mathcal{Q}}$ be the algebras
of $G_{r}$-invariant meromorphic rational functions on the complexification
$W_{{\mathbb{C}}}=W\otimes_{\mathbb{R}}{\mathbb{C}}$, which are
$K_{{\mathbb{C}}}$-finite and $\widetilde{K}_{{\mathbb{C}}}$-finite
respectively. By Lemmas \ref{Lem:Rm} and \ref{Lem: Rmm} each element of
$\mathcal{R}$ is a rational function on $W$, and hence extends uniquely to a
meromorphic rational function on $W_{{\mathbb{C}}}$. This gives us the first
map in the diagram below%
\[
\mathcal{R\longrightarrow Q}\overset{\sigma}{\longrightarrow}%
\widetilde{\mathcal{Q}}\longrightarrow\widetilde{\mathcal{R}},
\]
the middle map is obtained by the natural holomorphic action of $\sigma\in
G_{\mathbb{C}}$ on $W_{{\mathbb{C}}}$, while the last is the restriction map
to $\widetilde{X}$. All maps are injective and applying them to the
irreducible $K$-submodule $\mathcal{R}_{\mu}\subset\mathcal{R}$ as in Lemma
\ref{Lem:Rm}, we obtain an irreducible finite-dimensional $\widetilde{K}%
$-submodule $\widetilde{\mathcal{R}}_{\mu}$ whose $M$-fixed vector can be
identified with $\phi_{\tilde{\mu}}.$

The polynomials $\Psi$ and $\widetilde{\Psi}$ extend to holomorphic
polynomials on $W_{{\mathbb{C}}}$, while $C_{s,l}$ and $\tilde{C}_{s,l}$
extend to holomorphic differential operators on $W_{{\mathbb{C}}}.$ The action
of $\sigma$ on $W_{{\mathbb{C}}}$ carries $\Psi$ to $\widetilde{\Psi}$ and
$C_{s,l}$ to $\tilde{C}_{s,l}$. It follows that $\sigma$ intertwines the
action of $C_{s,l}$ on $\mathcal{R}_{\mu}$ with the action of $\tilde{C}%
_{s,l}$ on $\widetilde{\mathcal{R}}_{\mu}.$ Thus by Theorem \ref{Th:main1}
$\tilde{C}_{s,l}$ acts on $\widetilde{\mathcal{R}}_{\mu}$, and in particular
on $\phi_{\tilde{\mu}},$ by the scalar $q_{l}(2s-\rho_{1},\mu+\rho)$.
\end{proof}

\section{Nonspherical cases\label{section-non-spherical}}

We now extend the previous results to non-spherical line bundles over the
Grassmannian $Y=Gr_{n,r}(\mathbb{F})$. For $\mathbb{F}=\mathbb{R}$ the
nonspherical bundles are obtained by twisting by the sign
character\footnote{For $r=2$ the Grassmannian $Gr_{n,2}(\mathbb{R})$ does have
a complex structure and corresponding holomorphic line bundles, but this case
does not seem to fit into our framework and we shall not treat it here.}, and
have already been treated in \cite{sahi-2013-cap}. For $\mathbb{F}=\mathbb{H}$
there are no such bundles, therefore it remains only to consider
$\mathbb{F}=\mathbb{C}$.

The Grassmannian for $\mathbb{F}=\mathbb{C}$ is the compact Hermitian
symmetric space
\[
Y=Gr_{n,r}(\mathbb{C})=X/G_{r}=K/M=U(n)/U(r)\times U(n-r).
\]
For each integer $p$, the character $g\mapsto\operatorname{det}\left(
g\right)  ^{p}$ of $G_{r}=GL_{r}(\mathbb{C})$ induces a holomorphic line
bundle on $Y=X/G_{r}$, whose space of \emph{smooth} sections is given as
follows
\[
C^{\infty}(Y,p)=\left\{  f\in C^{\infty}\left(  X\right)
:f(xg)=\operatorname{det}(g)^{p}f(x)\text{ for all }g\in G_{r}\right\}  .
\]
As a $K$-module we have%
\[
C^{\infty}(Y,p)=\left\{  f\in C^{\infty}\left(  K\right)  :f(km)=\chi
_{p}\left(  m\right)  f(k)\text{ for all }m\in M\right\}
\]
where $\chi_{p}$ is the character of $M=U(r)\times U(n-r)$ given by $\chi
_{p}\left(  m_{1},m_{2}\right)  =\operatorname{det}\left(  m_{1}\right)  ^{p}$.

The irreducible $K$-submodules of $C^{\infty}(Y,p)$ occur with multiplicity
$1$ and parametrized explicitly in \cite{schlicht}. For $\Lambda,\Lambda^{m}$
as in (\ref{=Lam}, \ref{=Lam-m}) and for an integer $b\geq0$ we define
\begin{equation}
\Lambda_{b}=\{\lambda+\left(  b,\ldots,b\right)  \mid\lambda\in\Lambda
\},\quad\Lambda_{b}^{m}=\{\lambda+\left(  b,\ldots,b\right)  \mid\lambda
\in\Lambda^{m}\}, \label{=Lam-mp}%
\end{equation}
regard as subsets of $\mathfrak{a}^{\ast}$ as before. Then one has the
following result \cite[Th 7.2]{schlicht}.

\begin{lemma}
\label{schli} For each $\mu$ in $\Lambda_{\left\vert p\right\vert }$ there is
a unique $K$-type $V_{\mu}$ in $C^{\infty}(Y,p)$ whose highest weight
restricts to $\mu$ on $\mathfrak{a}$.
\end{lemma}

One also has an exact analog of the Harish-Chandra homomorphism
\cite{Shimura-1990, Shimeno-1990}.

\begin{lemma}
\label{Hcp}Let $D$ be a $K$-invariant differential operator on $C^{\infty
}(Y,p)$. There is a unique $\mathcal{W}_{r}$-invariant polynomial $p_{D}$ of
degree $ord\left(  D\right)  $ such that the eigenvalue of $D$ on $V_{\mu}$ is
$p_{D}\left(  \mu+\rho\right)  $.
\end{lemma}

In our setting, with $\Psi\left(  w\right)  =\operatorname{det}\left(
w^{\dagger}w\right)  $ and $L=\partial\left(  \Psi\right)  $ as in
(\ref{=Psi}), the operator
\[
C_{s,l}=\Psi^{l+s}L^{l}\Psi^{-s}%
\]
descends to a $K$-invariant differential operator of order $lr$ on $C^{\infty
}(Y,p)$. We will show that its eigenvalues can be expressed in terms of the
$\mathcal{W}_{r}$-invariant polynomial
\[
q_{l}\left(  t,z\right)  =\prod\nolimits_{i=0}^{l-1}\prod\nolimits_{j=1}%
^{r}\left[  \left(  t+2i\right)  ^{2}-z_{j}^{2}\right]  .
\]

We first construct an analog of the space $\mathcal{P}^{m}$. Let
$\mathcal{R}^{m}$ be the space of holomorphic polynomials $f$ on
$W=Mat_{n\times r}(\mathbb{C})$ satisfying
\begin{equation}
f(wg)=\operatorname{det}(g)^{m}f(w)\text{ for }g\in G_{r}, \label{R-m}%
\end{equation}
and let $\mathcal{R}^{m+b,m}$ be the space of sesqui-holomorphic polynomials
$f$ on $W\times\overline{W}$ satisfying
\[
f(w_{1}g_{1},w_{2}g_{2})=\operatorname{det}(g_{1})^{m+b}\overline
{\operatorname{det}(g_{2})^{m}}f(w_{1},w_{2})\text{ for }g_{1},g_{2}\in
G_{r}\text{.}%
\]
We will describe the $K$-type structure of $\mathcal{R}^{m+b,m}$ following
\cite{gkz-JLT}. We consider two subgroups of $G=GL(n,\mathbb{C}),$ written as
$(r,n-r)\times(r,n-r)$ block matrices, and a character
\[
P=\left\{
\begin{bmatrix}
a_{11} & a_{12}\\
0 & a_{22}%
\end{bmatrix}
\right\}  ,\quad Q=\left\{
\begin{bmatrix}
1 & a_{12}\\
0 & a_{22}%
\end{bmatrix}
\right\}  ,\quad\chi_{m}\left(
\begin{bmatrix}
a_{11} & a_{12}\\
0 & a_{22}%
\end{bmatrix}
\right)  =\operatorname{det}\left(  a_{11}\right)  ^{m}.
\]
Then $Q$ is the stabilizer of $x_{0}\in X$ (\ref{=x0}) and so we have natural
isomorphisms
\[
X\approx G/Q,\quad Y\approx X/G_{r}\approx G/P\text{.}%
\]
Thus if $\mathcal{L}_{m}$ is the line bundle on $G/P$ induced by the character
$\chi_{m}$ then polynomials in $\mathcal{R}^{m}$, after restriction to the
open set $X$, can be identified as holomorphic sections of $\mathcal{L}_{m}$.

We realize $C^{\infty}(Y,b)$ as the space of smooth functions on $X$
transforming as $f(xg)=\operatorname{det}^{b}(g)f(x)$, $g\in G_{r}.$

\begin{lemma}
\label{jlt} If $b\geq0$ then the map $J_{m}:f(x,y)\mapsto\operatorname{det}%
(x^{\dagger}x)^{-m}f(x,x)$ defines a $K$-isomorphism from $\mathcal{R}%
^{m+b,m}$ into a subspace of $C^{\infty}(Y,b)$ and we have $\mathcal{R}%
^{m+b,m}\approx\sum_{\mu\in\Lambda_{b}^{m}}V_{\mu}$ under the diagonal action
of $K$.
\end{lemma}

\begin{proof}
It is easy to see that the map $f\otimes g\mapsto F$, $F(w_{1},w_{2}%
)=f(w_{1})g(w_{2})$ is a $K$-equivariant isomorphism from $\mathcal{R}%
^{m+b}\otimes\overline{\mathcal{R}^{m}}$ to $\mathcal{R}^{m+b,m}$. The result
follows from the decomposition of the tensor product in \cite[Theorem
5.3]{gkz-JLT}. (Note that in \cite{gkz-JLT}, $\mathcal{R}^{m}$ is denoted
$A^{m,2}$ and $e_{j},\mu_{j}$ are denoted $\frac{1}{2}\beta_{j},$ $2m_{j}$.)
\end{proof}

\begin{theorem}
\label{line-bdle-cpt} The eigenvalue of $C_{s,l}$ on $V_{\mu}\subset L^{2}%
(Y,${$p$}$)$ is $q_{l}(2s-\rho_{1}- p,\mu+\rho)$.
\end{theorem}

\begin{proof}
The proof is very similar to Theorem \ref{Th:main1} and we sketch the main
arguments. We assume first that $p=b\geq0$. As in Theorem \ref{Th:main1} it
suffices to prove the result for $l=1$. By the Harish-Chandra homomorphism of
Lemma \ref{Hcp} and the argument of Lemma \ref{Lem: p}, the eigenvalues
$c_{s,\mu}$ of $C_{s}=C_{s,1}$ are given by a $\mathcal{W}_{r}$-invariant
polynomial of degree $2r$ in $\left(  s,\mu+\rho\right)  $. We first prove
that $c_{-m,\mu}$ vanishes if $2m=\mu_{1}-b$.

For this we note that $C_{-m}=\Psi^{-m+1}L\Psi^{m}$ and hence
\[
C_{-m}J_{m}=J_{m-1}L.
\]

Now $L$ maps the space $\mathcal{R}^{m+b,m}$ to $\mathcal{R}^{m-1+b,m-1}$ and
if $\mu_{1}=2m+b$ then by Lemma \ref{jlt} the $K$-type $\mu$ occurs in
$\mathcal{R}^{m+b,m}$ but not in $\mathcal{R}^{m-1+b,m-1}$, therefore $L$ acts
by $0$ on this $K$-type. Hence $C_{-m}J_{m}$ acts by $0$, since $J_{m}$ is an
isomorphism on this $K$-type, $C_{-m}$ acts by $0$ on $V_{\mu}$ and so
$c_{-m,\mu}=0.$

Finally by $\mathcal{W}_{r}$-invariance of $c_{s,\mu}$ it follows that up to
an overall constant multiple we have%
\begin{equation}
c_{s,\mu}=\prod\nolimits_{j=1}^{r}\left[  \left(  2s-\rho_{1}-p\right)
^{2}-\left(  \mu_{j}+\rho_{j}\right)  ^{2}\right]  =q_{1}\left(  2s-\rho
_{1}-p,\mu+\rho\right)  \text{.} \label{=csmup}%
\end{equation}
As before, equality follows by an easy calculation of leading coefficients.

For $p<0$, we realize $C^{\infty}(Y,\pm p)$ as function spaces on $X$ as above
and consider the following operator $T$ acting on functions on $X$
\[
Tf(x)=\operatorname{det}(x^{\dagger}x)^{p}f(\bar{x}).
\]
It is easy to see that $T$ is a $K$-isomorphism from $C^{\infty}\left(
Y,-p\right)  $ to $C^{\infty}\left(  Y,p\right)  $. Moreover, the function
$\Psi\left(  x\right)  =\operatorname{det}\left(  x^{\dagger}x\right)  $ and
the operator $L=\partial\left(  \Psi\right)  $ are invariant under conjugation
$x\mapsto\bar{x}$, and it follows that we have%
\[
TC_{s,l}=C_{s-p,l}T\text{.}%
\]
By the result for $p>0$, the eigenvalue of $C_{s,l}$ on $V_{\mu}$ in
$C^{\infty}\left(  Y,-p\right)  $ is
\[
q_{l}(2(s-p)-\rho_{1}-(-p),\mu+\rho)=q_{l}(2s-\rho_{1}-p,\mu+\rho)
\]
as claimed.
\end{proof}

We also consider the non-compact dual $\tilde{Y}$ of $Y$, and the space
\begin{equation}
C^{\infty}(\tilde{Y},p)=\left\{  f:Z\rightarrow\mathbb{C}\mid
f(zg)=\operatorname{det}(g)^{p}f(z)\text{ for all }g\in G_{r}\right\}
\label{line-bundle-sec}%
\end{equation}
where the open set $Z\subset W$ is as in (\ref{z-def}), then the operators
$C_{s,l}$ act as invariant differential operators on $C^{\infty}(\tilde{Y}%
,p)$. The $p$-spherical functions
\[
\phi_{\eta,p};\quad\eta\in\widetilde{\mathfrak{a}}^{\ast},p\in\mathbb{Z}%
\text{,}%
\]
as defined in \cite{Shimeno-jfa}, belong to the space $C^{\infty}(\tilde
{Y},p)$ \cite[Theorem 4.6]{gkz-JLT}, and are simultaneous eigenfunctions for
the $C_{s,l}$. As in Theorem \ref{thm-ncpt} one may compute the eigenvalues by
writing $\eta$ in the form $\widetilde{\lambda}$ for $\lambda\in
\mathfrak{a}^{\ast}$ as in (\ref{=lam-til}).

\begin{theo+}
For all $\lambda\in\mathfrak{a}^{\ast}$ we have $\tilde{C}_{s,l}\left(
\phi_{\widetilde{\lambda},p}\right)  =q_{l}(2s-\rho_{1}-p,\lambda+\rho
)\phi_{\widetilde{\lambda},p}$.
\end{theo+}

\begin{proof}
This is proved by exactly the same method as Theorem \ref{thm-ncpt} using
Theorem \ref{line-bdle-cpt}.
\end{proof}

\begin{rema+}
For any $p$ there is a $\tilde{K}_{n}=U(r,n-r)$-invariant $L^{2}$ space of
sections of the corresponding line bundle. For negative integer $p\leq-n$ the
$L^{2}$-space of holomorphic sections forms the holomorphic discrete series,
and it is generated by the spherical function $\phi_{\lambda,p}$ at
$\lambda=-\rho-(p,\cdots,p)$; see \cite{Shimeno-jfa}. In particular it is
annihilated by $C_{0}$ which is closely related to Shimura operators
\cite{gz-shimura}; see the remark below.
\end{rema+}

\begin{rema+}
Let $n=2r$ and $p=0$. The symmetric space $\tilde{Y}$ is the tube domain
$U(n,n)/U(n)\times U(n)$. The space $C^{\infty}(\tilde{Y})$ is now realized as
homogeneous functions on the subset $Z$ of $2r\times r$-matrices
$z=[z_{1},z_{2}]^{t}$ satisfying (\ref{line-bundle-sec}) for $p=0$. Let now
$s=-(r-1)$. Our results above claim that the differential
operator\footnote{Here we use the usual convention for complex
differentiation, $\partial_{w}=\frac{1}{2}(\partial_{u}-i\partial_{v})$ for a
complex variable $w=u+iv$. The appearance of the coefficient $2^{2r}$ is due
to the fact that our isomorphism (\ref{=ptl}) is done through the underlying
Euclidean inner product}%

\[
\tilde{C}_{s}=2^{2r}\operatorname{det}(z_{1}^{\ast}z_{1}-z_{2}^{\ast}%
z_{2})^{r}\operatorname{det}(\partial_{1}^{\ast}\partial_{1}-\partial
_{2}^{\ast}\partial_{2})^{r}\operatorname{det}(z_{1}^{\ast}z_{1}-z_{2}^{\ast
}z_{2})^{1-r}%
\]
has eigenvalue
\[
\prod\nolimits_{j=1}^{r}(1-(i\lambda_{j})^{2}).
\]
The operator $\tilde{C}_{s}$ is up to a non-zero constant, the Shimura
operator $\mathcal{L}_{1^{r}}$ \cite{Shimura-1990, gz-shimura}, as can be seen
by comparing $\tilde{C}_{s}$ with the explicit formula for $\mathcal{L}%
_{1^{r}}$ in the Siegel domain realization \cite[formula (3.8)]{gz-shimura}.
The eigenvalue for $\tilde{C}_{s}$ is also a consequence of \cite[Th.
3.4]{gz-shimura}, noticing again here we are using different basis, our basis
vectors $\{e_{j}\}$ are $\{\frac{1}{2}\beta_{j}\}$ there.
\end{rema+}

\section{Inverting the Radon transform\label{section-Radon}}

We shall apply our results above to find explicit differential operators
inverting the Radon transform. To begin with, we fix integers $r,r^{\prime}$
such that
\begin{equation}
r\leq n-r,\quad r\leq r^{\prime}, \label{eq:r-r'}%
\end{equation}
and we consider the corresponding Grassmannian manifolds
\[
Y=Gr_{n,r}(\mathbb{F})=K_{n}/K_{r}\times K_{n-r},\;Y^{\prime}=Gr_{n,r^{\prime
}}(\mathbb{F})=K_{n}/K_{r^{\prime}}\times K_{n-r^{\prime}}%
\]
We write $y$ and $\eta$ for typical elements of $Y$ and $Y^{\prime}$ and
define the incidence sets
\[
A_{\eta}=\left\{  y\in Y:y\subset\eta\right\}  \text{ and }B_{y}=\left\{
\eta\in Y^{\prime}:y\subset\eta\right\}  .
\]
These are homogeneous spaces for the stabilizers $K_{\eta}$ and $K_{y}$ of
$\eta$ and $y$ in $K=K_{n}$, and hence admit unique invariant probability
measures $dy$ and $d\eta$. The Radon transforms $R=R_{r^{\prime},r}:C^{\infty
}(Y)\rightarrow C^{\infty}(Y^{\prime})$ and $R^{\prime}=R_{r,r^{\prime}}%
:{C}^{\infty}(Y^{\prime})\rightarrow C^{\infty}(Y)$ are defined to be
\begin{equation}
Rf(\eta)=\int_{A_{\eta}}f(y)dy,\quad R^{\prime}F(y)=\int_{B_{y}}F(\eta)d\eta.
\label{rad-def}%
\end{equation}
Let $V_{\mu}$ be a $K_{n}$-isotypic subspace of $C^{\infty}(Y)$ as in Sect.
\ref{sec-K}. Since $R^{\prime}R$ is a $K_{n}$-invariant operator on $Y$, by
Schur's lemma there exist scalar eigenvalues $\gamma_{\mu}=\gamma_{\mu
}(r,r^{\prime},n)$ such that
\begin{equation}
R^{\prime}Rv=\gamma_{\mu}v\text{ for all }v\in V_{\mu}, \label{eq:gamma-mu}%
\end{equation}
We first describe an explicit formula for these eigenvalues $\gamma_{\mu}$.
For this we set%
\begin{equation}
\alpha=\frac{d}{2}r,\,\beta=\frac{d}{2}(n-r^{\prime}),l=\frac{d}{2}(r^{\prime
}-r)\text{;} \label{eq:abl}%
\end{equation}
let $\left(  c\right)  _{m}=c(c+1)\cdots(c+m-1)$ denote the Pochammer symbol,
and for $\mu$ define
\begin{align*}
m_{i}  &  =\mu_{i}/2,\quad\mathbf{m}=(m_{1},\ldots,m_{r}),\quad|\mathbf{m}%
|=m_{1}+\cdots+m_{r},\\
(c)_{\mathbf{m}}  &  =\prod\nolimits_{j=0}^{r-1}\left(  c-dj/2\right)
_{m_{j+1}}\text{,\quad}\eta_{\mu}\left(  \nu\right)  =\frac{(-1)^{|\mathbf{m}%
|}(-\nu)_{\mathbf{m}}}{(dn/2+\nu)_{\mathbf{m}}}.
\end{align*}

\begin{theorem}
\label{gam-mu}For $n,r,r^{\prime}$ as in (\ref{eq:r-r'}), we have$.$
\begin{equation}
\gamma_{\mu}=\eta_{\mu}\left(  -\alpha\right)  \eta_{\mu}\left(
-\beta\right)  =\dfrac{(\alpha)_{\mathbf{m}}(\beta)_{\mathbf{m}}}%
{(\alpha+l)_{\mathbf{m}}(\beta+l)_{\mathbf{m}}}. \label{=gam-mu}%
\end{equation}

\end{theorem}

This result is proved in \cite{Grinberg-jdg} for the complex case, and stated
without proof for the real and quaternionic cases. While the real case can
indeed be proved similarly, it seems to us that the quaternionic case requires
substantially different arguments. Moreover the statement in \cite[Theorem
6.2]{Grinberg-jdg} for the quaternionic case is wrong. As an immediate test,
the formula there cannot be correct because it vanishes on the constant
function. To be more precise we note first that all the factors in the
displayed formula should be inverted, and the factors should be
\[
(m_{j}+2(k+1)-2j)(m_{j}+2(k+1)-2j+1)(m_{j}+2(n-k-1)-2j)(m_{j}+2(n-k-1)-2j+1).
\]
The constant $k$, $n$ and $m_{j},j=0,\cdots,k$ there are our $r-1$, $n-1$ and
$\frac{\mu_{i}}{2}$, $i=1,\cdots,r$ respectively; Theorem 6.2
\cite{Grinberg-jdg} is stated for the Radon transform $R_{k,k+1}$ in the
notation there, namely for our $R_{r+1,r}$ with $r^{\prime}=r+1$ and $l=2$.

For the above reasons we give a different calculation of $\gamma_{\mu}$, using
ideas from \cite{gz-radon-imrn, gz-radon-bsd}, where the Radon transform is
studied along with the spherical transform of the sine functions on
Grassmannians and their non-compact duals, the real bounded symmetric domains.
It was proved that when $r\leq\frac{r^{\prime}}{2}$ the eigenvalues of the
Radon transform are given by the evaluation of spherical polynomials at
certain specific points up to a unknown factor; see Lemma 7.3 below. In the
present paper we first find an explicit formula for this special value.

We follow the notation in \cite{gz-radon-imrn} and let $\{E_{j}\}_{j=1}^{r}$
denote the basis of the Cartan subalgebra $\mathfrak{a}$ dual to
$\{e_{j}\}_{j=1}^{r}$. We write $K=K_{n}$ and $M=K_{r}\times K_{n-r}$ and
consider two points $y_{0}$ and $y_{1}$ in the space $Y=K/M$. The point
$y_{0}$ is the identity coset in $Y$ as in Section 2, while $y_{1}$ is the
coset of the element%

\[
\operatorname{exp}\frac{i\pi}{2}\left(  E_{1}+\cdots+E_{r}\right)
\]
Let $\phi_{\mu}$ be the $M$-spherical function in the space $V_{{\mu}}$, i.e.,
the unique element $M$-invariant function normalized so that $\phi_{\mu}%
(y_{0})=1$. As in \cite{gz-radon-imrn} let $|\operatorname{Sin}y|$ be the
$M$-invariant function on $Y$ whose restriction to $\mathfrak{a}$ is given by
the sine function.

The following lemma is proved in \cite{gz-radon-bsd} for the non-compact dual
symmetric spaces of $Y$ and $Y^{\prime}$, the proof for the compact case is
exactly the same. Let $\eta_{0}\in Y^{\prime}$ be the space $\eta
_{0}=[\mathbb{F}^{r^{\prime}}]$ with $y_{0}=[\mathbb{F}^{r}]$ the standard
subspace in $\eta_{0}$. Then the incidence set of $\eta_{0}$ is the symmetric
subspace of $Y$
\[
A_{\eta_{0}}=Gr_{r^{\prime},r}(\mathbb{F})=K_{r^{\prime}}/K_{r}\times
K_{r^{\prime}-r}\subset Y
\]
of rank $\text{min}(r,r^{\prime}-r)$. In particular if $r\leq\frac{r^{\prime}%
}{2}$ it is of the same rank as $Y$ with $\mathfrak{a}$ as a corresponding
Cartan subspace.

\begin{lemma}
Let $f$ be a $M$-invariant smooth function on $Y$.

\begin{enumerate}
\item {For $n,r,r^{\prime}$ as in (\ref{eq:r-r'}) }we have
\begin{equation}
(R^{\prime}Rf)(y_{0})=Rf(\eta_{0})=\int_{A_{\eta_{0}}}f(y)dy.
\label{radon-ver-sine-2}%
\end{equation}

\item If furthermore $r\leq\frac{r^{\prime}}{2}$ then this can be written as
an integration on $Y$,
\begin{equation}
(R^{\prime}Rf)(y_{0})=Rf(\eta_{0})=\int_{Y}|\operatorname{Sin}y|^{-2\beta
}f(y)dy, \label{radon-ver-sine}%
\end{equation}
where $\beta$ is as in (\ref{eq:abl}) and $dy$ is the $K$-invariant measure
and $|\operatorname{Sin}y|^{-2\beta}dy $ is an $M$-invariant probability
measure on $Y$.
\end{enumerate}
\end{lemma}

We next need a couple of rationality results for $\phi_{{\mu}}(y_{1})$ and
$\gamma_{{\mu}}$.

\begin{lemma}
For fixed $r$ and ${\mu}$, $\phi_{{\mu}}(y_{1})$ is a rational function of $n$.
\end{lemma}

\begin{proof}
It is well known (see e.g. \cite[Theorem 9.1]{Koor}) that $\phi_{\mu
}(\operatorname{exp}(it_{1}E_{1}+\cdots+it_{r}E_{r}))$ is a polynomial in
${\sin t_{1},\cdots,\sin t_{r}}$ of degree $2|\mu|$, with coefficients
depending rationally on the root multiplicities, hence on $n$ for fixed $r$.
In particular the special value at $y_{1}$ depends rational on $n.$
\end{proof}

\begin{lemma}
For fixed $r$ and $\mu$, $\gamma_{{\mu}}\left(  n,r,r^{\prime}\right)  $ is a
rational function of $n$ and $r^{\prime}$.
\end{lemma}

\begin{proof}
For $r<\frac{r^{\prime}}{2}$ we can write $\gamma_{{\mu}}\left(
n,r,r^{\prime}\right)  $ an integral over $\mathfrak{a}$. Explicitly we have
\[%
\begin{split}
\gamma_{{\mu}}\left(  n,r,r^{\prime}\right)   &  =\int_{Y}| \operatorname{Sin}
y|^{-2\beta} \phi_{\mu}(y)dy\\
&  =\int_{Q_{0}}\phi_{\mu}(\operatorname{exp} (it_{1}E_{1}+\cdots+it_{r}%
E_{r}))\prod_{\alpha\in R_{+}}|\sin\alpha(it)|^{m_{\alpha}^{\prime}}%
dt_{1}\cdots dt_{r};
\end{split}
\]
see also \cite[(2.5), (3.11)]{gz-radon-imrn} with the integration being
normalized as well. Here $m_{\alpha}^{\prime}=m_{\alpha}-2\beta
=d(n-2r)-d(n-r^{\prime}) =d(r^{\prime}-2r)> 0$ if $\alpha$ is the root $e_{j}$
and $m_{\alpha}^{\prime}=m_{\alpha}$ other wise, with $m_{\alpha}$ being the
restricted root multiplicity specified in Section 3, and $Q_{0}$ is a
fundamental polygon in $\mathfrak{a}$. Now there exist two orthogonal bases
$\phi_{\mu}$ and $\phi_{\mu^{\prime}}^{\prime}$, which are the Heckman-Opdam
polynomials \cite{Opdam-acta}, with respect to the root multiplicity function
$m_{\alpha}$ and $m_{\alpha}^{\prime}$ respectively, all normalized by
$\phi(0)=1$. The integral above is precisely the constant term in the
expansion of $\phi_{\mu}$ in terms of $\phi_{\mu^{\prime}}^{\prime}$. But both
$\phi_{\mu}$ and $\phi_{\mu^{\prime}}^{\prime}$ have triangular expansions in
terms of the monomials with coefficients that are rational in both $m_{\alpha
}$ and $m_{\alpha}^{\prime}$ respectively, hence rational in $n, r^{\prime}$;
see \cite[Theorem 9.1]{Koor}, which in turn implies that $\phi_{\mu}$ can be
expanded in terms of $\phi_{\mu^{\prime}}^{\prime}$ with rational
coefficients. This proves the Lemma.
\end{proof}

Thus in the proofs below we may assume, if necessary, that
\begin{equation}
\label{large-r}n>>r^{\prime}>> r.
\end{equation}

\begin{lemma}
We have $\gamma_{\mu}=\eta_{\mu}\left(  -\beta\right)  \phi_{\mu}\left(
y_{1}\right)  $.
\end{lemma}

\begin{proof}
We may assume $r^{\prime}> 2r$. By (7.2) and (7.6) we have $R^{\prime}%
R\phi_{\mu}(y_{0}) =\gamma_{\mu}$ is as an integration of $\phi_{\mu}$ on $M$
against the function $|\text{Sin}|^{-2\beta}$, and it is shown in
\cite{gz-radon-imrn} that for any $\nu$%
\begin{equation}
\int_{Y}|\operatorname{Sin}y|^{2\nu}\phi_{\mu}(y)dy=\eta_{\mu}(\nu)\phi_{\mu
}(y_{1}) \label{sph-sin}%
\end{equation}
The result now follows by applying (\ref{radon-ver-sine}) to $\phi_{\mu}$ and
using (\ref{sph-sin}).
\end{proof}

The value $\phi_{\mu}(y_{1})$ was only found \cite{gz-radon-imrn} for certain
symmetric spaces. In the present paper we shall find an evaluation formula by
using the interpretation in \cite{gz-radon-imrn} of the Gelfand integral
formula \cite[Ch. IV, Sect. 2, Proposition 2.2]{He2} as a Radon transform.
This formula might be also of independent interest. The reader may also
consult \cite[Ch. I, Sect. 4, Lemma 4.9]{He2} and \cite[Ch. I, Sect. 3,
Remark, pp.55-56]{He3} for the relevance of this formula in the study of Radon transform.

\begin{proposition}
We have $\phi_{\mu}(y_{1})=\eta_{\mu}\left(  -\alpha\right)  $, where $\alpha$
is as in Theorem 7.1.
\end{proposition}

\begin{proof}
We may assume that $r\leq\frac{n}{3}$ and follow the proof of \cite[Lemma
4.5]{gz-radon-imrn}, which shows that
\begin{equation}
\phi_{\mu}(y_{1})^{2}=R_{r,n-r}R_{n-r,r}\phi_{\mu}(y_{0}). \label{square-phi}%
\end{equation}
Now by the previous Lemmas 7.2 and 7.5, with $r^{\prime}$ replaced by $n-r$
and $\beta$ by $\alpha$, we get
\[
R_{r,n-r}R_{n-r,r}\phi_{\mu}(y_{0})=\left[  \eta_{\mu}\left(  -\alpha\right)
\phi_{\mu}\left(  y_{1}\right)  \right]  \phi_{\mu}(y_{0})=\eta_{\mu}\left(
-\alpha\right)  \phi_{\mu}\left(  y_{1}\right)
\]
Combining these formulas we get
\[
\phi_{\mu}(y_{1})^{2} =\eta_{\mu}\left(  -\alpha\right)  \phi_{\mu}\left(
y_{1}\right)  ,
\]
and it suffices to show that $\phi_{\mu}\left(  y_{1}\right)  \neq0$. However
for $r\leq n/2$ the map $R_{r,n-r}R_{n-r,r}$ is injective, which follows for
example from the analytical inversion formulae in \cite{grinberg-rubin-ann,
gz-radon-djm}, and thus by (\ref{square-phi}) we must have $\phi_{\mu}%
(y_{1})\neq0$
\end{proof}

We now prove Theorem \ref{gam-mu}.

\begin{proof}
The identity $\gamma_{\mu}=\eta_{\mu}\left(  -\alpha\right)  \eta_{\mu}\left(
-\beta\right)  $ follows directly from Lemma 7.5 and Proposition 7.6. To
rewrite this in the form $\dfrac{(\alpha)_{\mathbf{m}}(\beta)_{\mathbf{m}}%
}{(\alpha+l)_{\mathbf{m}}(\beta+l)_{\mathbf{m}}}$ we recall the definition of
$\eta_{\mu}\left(  -\nu\right)  $ and use the identity $dn/2=\alpha+\beta+l$.
\end{proof}

We will use Theorem \ref{gam-mu} to find an inverse for the Radon transform
under the assumptions
\begin{equation}
r\leq r^{\prime}\leq n-r; \label{con-1}%
\end{equation}
and the following parity condition, which is automatic for $\mathbb{F}%
=\mathbb{C},\mathbb{H}$,
\begin{equation}
l=\frac{d}{2}(r^{\prime}-r){\,\text{is an integer}}\text{.} \label{con-2}%
\end{equation}
We show that in this setting $\gamma_{\mu}^{-1}$ is a \emph{symmetric
polynomial} in $\bar{\mu}=\mu+\rho$. More precisely, we recall from Theorem
\ref{Th:Cs} that the differential operator $C_{s}$ has eigenvalues%
\[
c_{s,\mu}=\prod\nolimits_{j=1}^{r}\left[  (2s+\rho_{1})^{2}-(\mu_{j}+\rho
_{j})^{2}\right]
\]
and we will relate the eigenvalue $\gamma_{\mu}$ to the product
\[
\varepsilon_{\mu}=\prod\nolimits_{i=0}^{l-1}c_{-\beta-i,\mu}=\prod
\nolimits_{j=1}^{r}\prod\nolimits_{i=0}^{l-1}\left[  \left(  -2\beta
-2i+\rho_{1}\right)  ^{2}-\bar{\mu}_{j}^{2}\right]
\]

\begin{lemma}
Under the parity assumption (\ref{con-2}) we have $\gamma_{\mu}^{-1}%
=\varepsilon_{\mu}/\varepsilon_{0}$.
\end{lemma}

\begin{proof}
If $a_{\mu}$ is a function of $\mu$ and $S$ is a multiset of real numbers, we
will write
\[
a_{\mu}\sim S\text{ iff }a_{\mu}=c\prod_{j=1}^{r}\prod_{s\in S}\left(
s-\bar{\mu}_{j}\right)  \text{ for some constant }c\text{ independent of }\mu
\]
Using $\rho_{j}-\rho_{1}=d\left(  j-1\right)  $ and the identity$\frac{\left(
c+l\right)  _{m}}{\left(  c\right)  _{m}}=\frac{\Gamma\left(  c+l+m\right)
\Gamma\left(  c\right)  }{\Gamma\left(  c+m\right)  \Gamma\left(  c+l\right)
}=\frac{\left(  c+m\right)  _{l}}{\left(  c\right)  _{l}}$, we get
\[
\tfrac{(\alpha+l)_{\mathbf{m}}}{(\alpha)_{\mathbf{m}}}=\prod_{j=1}^{r}%
\dfrac{\left(  \alpha+l+\frac{\rho_{j}-\rho_{1}}{2}\right)  _{m_{j}}}{\left(
\alpha+\frac{\rho_{j}-\rho_{1}}{2}\right)  _{m_{j}}}=\prod_{j=1}^{r}%
\dfrac{\left(  \alpha+m_{j}+\frac{\rho_{j}-\rho_{1}}{2}\right)  _{l}}{\left(
\alpha+\frac{\rho_{j}-\rho_{1}}{2}\right)  _{l}}=c\prod_{j=1}^{r}\left(
\frac{2\alpha-\rho_{1}+\bar{\mu}_{j}}{2}\right)  _{l}.
\]

Thus writing $\alpha_{i}=-2\alpha-2i+\rho_{1}$ and $\beta_{i}=2\beta
+2i-\rho_{1}$ we have
\begin{equation}
\tfrac{(\alpha+l)_{\mathbf{m}}}{(\alpha)_{\mathbf{m}}}\sim\left\{  -\alpha
_{0},\ldots,-\alpha_{l-1}\right\}  ,\quad\tfrac{(\beta+l)_{\mathbf{m}}}%
{(\beta)_{\mathbf{m}}}\sim\left\{  -\beta_{0},\ldots,-\beta_{l-1}\right\}  .
\label{=alphabeta}%
\end{equation}
We now use the identity $\alpha+\beta=dn/2-l$ to deduce that%
\[
\alpha_{i-1}+\beta_{l-i}=-2\left(  \alpha+\beta\right)  +2\rho_{1}%
-2l+2=-dn+2\rho_{1}-2=0\text{.}%
\]
Thus we have $-\alpha_{i-1}=\beta_{l-i}$ and now by (\ref{=alphabeta}) we get
\[
\gamma_{\mu}^{-1}=\tfrac{(\alpha+l)_{\mathbf{m}}}{(\alpha)_{\mathbf{m}}}%
\tfrac{(\beta+l)_{\mathbf{m}}}{(\beta)_{\mathbf{m}}}\sim\left\{  \pm\beta
_{0},\ldots,\pm\beta_{l-1}\right\}  .
\]
It follows that there is a constant $c$ such that%
\[
\gamma_{\mu}^{-1}=c\prod\nolimits_{j=1}^{r}\prod\nolimits_{i=0}^{l-1}\left(
\beta_{i}^{2}-\bar{\mu}_{j}^{2}\right)  =c\varepsilon_{\mu}.
\]
Since $\gamma_{0}=1$, we deduce $c=\varepsilon_{0}^{-1}$ and the result follows.
\end{proof}

Let $C_{s}$ be the differential operator from (\ref{eq:C}), and define%
\[
S=\varepsilon_{0}^{-1}\prod\nolimits_{i=0}^{l-1}C_{-\beta-i}=\varepsilon
_{0}^{-1}\Psi^{\beta+l}L^{l}\Psi^{-\beta}.
\]

\begin{theorem}
\label{Th: main2}Under the assumptions (\ref{con-1}) and (\ref{con-2}) we have
$SR^{\prime}R=I$.
\end{theorem}

\begin{proof}
It is enough to prove that $SR^{\prime}R$ acts on each $V_{\mu}$ by $1$. By
Theorem \ref{Th:Cs} the operator $S$ acts on $\mathcal{H}(\mu)$ by the scalar
eigenvalue
\begin{equation}
\varepsilon_{0}^{-1}\prod\nolimits_{i=0}^{l-1}c_{-\beta-i,\mu}=\varepsilon
_{0}^{-1}\prod\nolimits_{i=0}^{l-1}\left(  \beta_{i}^{2}-\bar{\mu}_{j}%
^{2}\right)  =\varepsilon_{\mu}/\varepsilon_{0}=\gamma_{\mu}^{-1}%
.\label{prod-c-eig}%
\end{equation}
Now by (\ref{eq:gamma-mu}) $SR^{\prime}R$ acts on $V_{\mu}$ by $\gamma_{\mu
}^{-1}\gamma_{\mu}=1.$
\end{proof}

\begin{rema+}
\label{RubinFunk} (Relation to the Funk transform.) Let $V=V_{n,r}$ and
$U=V_{n,k}$ be the Stiefel manifolds of orthonormal $r$-frames and $k$-frames
in $\mathbb{R}^{n}$, and for $u\in U$ define
\[
V^{u}:=\left\{  v\in V:u^{t}v=0\right\}  .
\]
This is a homogeneous space for the stabilizer of $u$ in $O\left(  n\right)  $
and hence carries a unique invariant probability measure $d_{u}v$. The Funk
transform $F=F_{r,k}$ is defined as follows
\[
F:C^{\infty}\left(  V\right)  \rightarrow C^{\infty}\left(  U\right)  ,\quad
Ff(u)=\int_{V^{u}}f(v)d_{u}v.
\]

In \cite[Theorem 8.2]{rubin-funk} Rubin obtains an inversion formula for $F$
restricted to $O\left(  r\right)  $-invariant functions under the conditions
\begin{equation}
r\leq k\leq n-r,\quad n-k-r\text{ is even.}\label{=rub-cond}%
\end{equation}

We now explain how to deduce Rubin's result from ours. Since $V^{u}=V^{ur}$
for all $g\in O\left(  k\right)  $ the image of $F$ is contained in $O\left(
k\right)  $-invariant functions. Moreover we have natural identifications
\[
C^{\infty}\left(  V_{n,r}\right)  ^{O\left(  r\right)  }\approx C^{\infty
}\left(  Gr_{n,r}\right)  ,\quad C^{\infty}\left(  V_{n,k}\right)  ^{O\left(
k\right)  }\approx C^{\infty}\left(  Gr_{n,k}\right)  \approx C^{\infty
}\left(  Gr_{n,n-k}\right)
\]
where the last isomorphism corresponds to taking orthogonal complements. The
restriction of $F$ to $O\left(  r\right)  $-invariant functions agrees with
the Radon transform $R=R_{r^{\prime},r}$ with $r^{\prime}=n-k$. We note also
that the conditions (\ref{=rub-cond}) are precisely the special case $d=1$ of
our conditions (\ref{=i-ii}) for $r^{\prime}=n-k$. The inversion formula of
\cite[Theorem 8.2]{rubin-funk} has the same form as ours, up to a certain
explicit constant $c_{r,k}$, and we now give a brief argument to show that
this is precisely $c_{\beta,l}\left(  \rho\right)  $ as in Theorem \ref{ThB}.
For this we recall the Gindikin Gamma function \cite{FK-book} associated with
the Jordan algebra $A$,
\[
\Gamma_{A}(x)=\prod_{j=1}^{r}\Gamma(x_{j}-\frac{d}{2}(j-1)),\quad
x=(x_{1},\cdots,x_{r}).
\]
For $\mathbf{m}=(m_{1},\ldots,m_{r})$ we put $z+\mathbf{m}=(z+m_{1}%
,\ldots,z+m_{r})$; then an easy calculation shows that for a $K$-type
$\mu=2\sum_{j=1}^{r}m_{j}e_{j}$ we have%
\[
c_{s,l}\left(  \rho+\mu\right)  =(-1)^{lr}2^{2lr}\frac{\Gamma_{A}%
(s+l+\mathbf{m})\Gamma_{A}(-s+\frac{dn}{2}+\mathbf{m})}{\Gamma_{A}%
(s+\mathbf{m})\Gamma_{A}(-s+\frac{dn}{2}-l+\mathbf{m})}.
\]
Specializing to $\mu=0$ and $s=\beta=\frac{d}{2}(n-r^{\prime})$ we get
\[
c_{\beta,\lambda}\left(  \rho\right)  =(-1)^{-lr}2^{-2lr}\frac{\Gamma
_{A}(\frac{d}{2}(n-r^{\prime}))\Gamma_{A}(\frac{d}{2}r)}{\Gamma_{A}(\frac
{d}{2}(n-r))\Gamma_{A}(\frac{d}{2}r^{\prime})}.
\]
For the real case $d=1$ this is precisely the constant $c_{r,k}$ \cite[Theorem
8.2]{rubin-funk}.
\end{rema+}

\end{document}